\documentclass[12pt]{amsart}

\usepackage{amssymb, amscd, txfonts}
\usepackage{graphicx}

\numberwithin{equation}{section}

\newtheorem{thm}{Theorem}[section]
\newtheorem{prop}[thm]{Proposition}
\newtheorem{lem}[thm]{Lemma}

\theoremstyle{definition}
\newtheorem{defn}[thm]{Definition}

\theoremstyle{remark}
\newtheorem{rem}[thm]{Remark}

\renewcommand{\ker}{\operatorname{Ker}}

\newcommand{\Z}{\mathbb{Z}}
\newcommand{\Q}{\mathbb{Q}}
\newcommand{\R}{\mathbb{R}}

\newcommand{\K}{\mathbb{K}}

\DeclareMathOperator{\im}{Im}

\DeclareMathOperator{\tr}{tr}

\begin{document}

\title[Torsion and Morse-Novikov theory]{Non-commutative Reidemeister torsion and Morse-Novikov theory}
\author[T.~Kitayama]{Takahiro KITAYAMA}
\address{Graduate~School~of~Mathematical~Sciences, the~University~of~Tokyo, 3-8-1~Komaba, Meguro-ku, Tokyo 153-8914, Japan}
\email{kitayama@ms.u-tokyo.ac.jp}

\begin{abstract}
Given a circle-valued Morse function of a closed oriented manifold, we prove that Reidemeister torsion over a non-commutative formal Laurent polynomial ring equals the product of a certain non-commutative Lefschetz-type zeta function and the algebraic torsion of the Novikov complex over the ring.
This paper gives a generalization of the result of Hutchings and Lee on abelian coefficients to the case of skew fields.
As a consequence we obtain a Morse theoretical and dynamical description of the higher-order Reidemeister torsion.
\end{abstract}

\maketitle

\section{Introduction}

In this paper let $X$ be a closed connected oriented Riemannian $d$-manifold with $\chi(X) = 0$ and $f \colon X \to S^1$ a Morse function such that the stable and unstable manifolds of the critical points of $f$ transversely intersect and the closed orbits of flows of $\nabla f$ are all nondegenerate.
(See Section \ref{sec_2_2} and \ref{sec_3_1}.)

For a generic closed $1$-form, for instance $df$, we can define the Lefschetz-type zeta function which counts closed orbits of flows induced by the $1$-form.
In \cite{Hu}, \cite{HL1}, \cite{HL2} Hutchings and Lee showed that the product of the zeta function and the algebraic torsion of the abelian Novikov complex associated to the $1$-form is a topological invariant and is equal to the abelian Reidemeister torsion of $X$.
In \cite{Paz2}, \cite{Paz3} Pazhitnov also proved a similar theorem in terms of the torsion of a canonical chain homotopy equivalence map between the abelian Novikov complex and the completed simplicial chain complex of the maximal abelian covering of $X$.
In the case where $X$ is a fiber bundle over a circle and $f$ is the projection these results give Milnor's theorem in \cite{M2}, which claims that the Lefschetz zeta function of a self map is equal to the abelian Reidemeister torsion of the mapping torus of the map.

In fixed point theory there is a non-commutative substitute for the Lefschetz zeta function which is called the total Lefschetz-Nielsen invariant, and in \cite{GN} Geoghegan and Nicas showed that the invariant has similar properties to these of torsion and determines the Reidemeister traces of iterates of a self map.
In \cite{Paj1} Pajitnov considered the eta function associated to $- \nabla f$ which lies in a suitable quotient of the Novikov ring of $\pi_1 X$ and whose abelianization coincides with the logarithm of the Lefschetz-type zeta function.
He also proved a formula expressing the eta function in terms of the torsion of a chain homotopy equivalence map between the Novikov complex and the completed simplicial chain complex of the universal covering of $X$.
These works were generalized to the case of generic closed $1$-forms by Sch\"utz in \cite{Sc1} and \cite{Sc2}. 

Non-commutative Alexander polynomials which are called the higher-order Alexander polynomials were introduced, in particular for $3$-manifolds, by Cochran in \cite{C} and Harvey in \cite{Ha}, and are known by Friedl in \cite{F} to be essentially equal to Reidemeister torsion over certain skew fields.
We call it higher-order Reidemeister torsion.
The aim of this paper is to give a generalization of Hutchings and Lee's theorem to the case where the coefficients are skew fields by using Dieudonn\'{e} determinant and to obtain a Morse theoretical and dynamical description of higher-order Reidemeister torsion.
Note that it is known by Goda and Pajitnov in \cite{GP} that the torsion of a chain homotopy equivalence between the twisted Novikov complex and the twisted simplicial complex by a linear representation equals the twisted Lefschetz zeta function which was introduced by Jiang and Wang in \cite{JW}.
This work is closely related to twisted Alexander polynomials which were introduced first by Lin in \cite{L} and later generally by Wada in \cite{W}. 
Our objects and approach considered here are different from theirs.

Let $\Lambda_f$ be the Novikov completion of $\Z[\pi_1 X]$ associated to  $f_* \colon \pi_1 X \to \pi_1 S^1$.
We first consider a certain quotient group $\overline{(\Lambda_f^{\times})}_{ab}$ of the abelianization of the unit group $\Lambda_f^{\times}$ and introduce non-commutative Lefschetz-type zeta function $\zeta_f \in \overline{(\Lambda_f^{\times})}_{ab}$ of $f$.
Taking a  poly-torsion-free-abelian group $G$ and group  homomorphisms $\rho \colon \pi_1 X \to G$, $\alpha \colon G \to \pi_1 S^1$ such that $\alpha \circ \rho = f_*$,  we construct a certain Novikov-type skew field $\mathcal{K}_{\theta}((t^l))$.
Similar to $\overline{(\Lambda_f^{\times})}_{ab}$ we define a certain quotient group $\overline{\mathcal{K}_{\theta}((t^l))}_{ab}^{\times}$ of the abelianization $\mathcal{K}_{\theta}((t^l))_{ab}^{\times}$ of $\mathcal{K}_{\theta}((t^l))^{\times}$.
We can check that $\rho$ naturally extends to a ring homomorphism $\Lambda_f \to \mathcal{K}_{\theta}((t^l))$ and also denote it by $\rho$.
There is a naturally induced homomorphism $\rho_* \colon \overline{(\Lambda_f^{\times})}_{ab} \to \overline{\mathcal{K}_{\theta}((t^l))}_{ab}^{\times}$ by $\rho$.
If the twisted homology group $H_*^{\rho}(X ; \mathcal{K}_{\theta}((t^l)))$ of $X$ associated to $\rho$ vanishes, then we can define the Reidemeister torsion $\tau_{\rho}(X)$ of $X$ associated to $\rho$ and the algebraic torsion $\tau_{\rho}^{Nov}(f)$ of the Novikov complex over $\mathcal{K}_{\theta}((t^l))$ as elements in $\mathcal{K}_{\theta}((t^l))_{ab}^{\times} / \pm \rho(\pi_1 X)$.
Here is the main theorem which can be applied for the higher-order Reidemeister torsion.

\begin{thm}[Theorem \ref{thm_M}]
For a given pair $(\rho, \alpha)$ as above, if $H_*^{\rho}(X ; \mathcal{K}_{\theta}((t^l))) = 0$, then
\[ \tau_{\rho}(X) = \rho_*(\zeta_f) \tau_{\rho}^{Nov}(f) \in \overline{\mathcal{K}_{\theta}((t^l))}_{ab}^{\times} / \pm \rho(\pi_1 X). \]
\end{thm}

To prove the theorem we use a similar approach to that of Hutchings and Lee in \cite{HL2}, but we need more subtle argument because of the non-commutative nature, especially in the second half, which is the heart of the proof.
We can check that the non-commutative zeta function $\zeta_f$ can be seen as a certain reduction of the eta function associated to $- \nabla f$, and this theorem can also be deduced from the results of Pajitnov in \cite{Paj1} by a purely algebraic functoriality argument.
In \cite[Theorem 5.4]{GS} Goda and Sakasai showed another splitting formula for Reidemeister torsion over skew fields.

This paper is organized as follows.
In the next section we review some of the standard facts of Reidemeister torsion and the Novikov complex of $f$.
In Section \ref{sec_3} we introduce the non-commutative Lefschetz-type zeta function $\zeta_f$ and construct the skew field $\mathcal{K}_{\theta}((t^l))$.
There we also set up notation for higher-order Reidemeister torsion.
Section \ref{sec_4} is devoted to the proof of the main theorem.

\section{Preliminaries}
\subsection{Reidemeister torsion}
We begin with the definition of Reidemeister torsion over a skew field $\K$.
See \cite{M1} and \cite{T1} for more details.

For a matrix over $\K$, we mean by an elementary row operation the addition of a left multiple of one row to another row.
After elementary row operations we can turn any matrix $A \in GL_k(\K)$ into a diagonal matrix $(d_{i, j})$.
Then the \textit{Dieudonn\'{e} determinant} $\det A$ is defined to be $[\prod_{i=1}^k d_{i, i}] \in \K_{ab}^{\times} := \K^{\times} / [\K^{\times}, \K^{\times}]$.

Let $C_* = (C_n \xrightarrow{\partial_n} C_{n-1} \to \cdots \to C_0)$ be a chain complex of finite dimensional right $\K$-vector spaces.
If we have bases $b_i$ of $\im \partial_{i+1}$ and $h_i$ of $H_i(C_*)$ for $i = 0, 1, \dots n$, we can take a basis $b_i h_i b_{i-1}$ of $C_i$ as follows.
Picking a lift of $h_i$ in $\ker \partial_i$ and combining it with $b_i$, we first obtain a basis $b_i h_i$ of $C_i$.
Then picking a lift of $b_{i-1}$ in $C_i$ and combining it with $b_i h_i$, we can obtain a basis $b_i h_i b_{i-1}$ of $C_i$.

\begin{defn}
For given bases $\boldsymbol{c} = \{ c_i \}$ of $C_*$ and $\boldsymbol{h} = \{ h_i \}$ of $H_*(C_*)$, we choose a basis $\{ b_i \}$ of $\im \partial_*$ and define
\[ \tau(C_*, \boldsymbol{c}, \boldsymbol{h}) := \prod_{i=0}^n [b_i h_i b_{i-1} / c_i]^{(-1)^{i+1}} ~ \in \K_{ab}^{\times}, \]
where $[b_i h_i b_{i-1} / c_i]$ is the Dieudonn\'{e} determinant of the base change matrix from $c_i$ to $b_i h_i b_{i-1}$.
If $C_*$ is acyclic, then we write $\tau(C_*, \boldsymbol{c})$.
\end{defn}

It can be easily checked that $\tau(C_*, \boldsymbol{c}, \boldsymbol{h})$ does not depend on the choices of $b_i$ and $b_i h_i b_{i-1}$.

Torsion has the following multiplicative property.
Let
\[ 0 \to C_*' \to C_* \to C_*'' \to 0 \]
be a short exact sequence of finite chain complexes of finite dimensional right $\K$-vector spaces and let $\boldsymbol{c} = \{ c_i \}, \boldsymbol{c}' = \{ c_i' \}, \boldsymbol{c}'' = \{ c_i'' \}$ and $\boldsymbol{h} = \{ h_i \}, \boldsymbol{h}' = \{ h_i' \}, \boldsymbol{h}'' = \{ h_i'' \}$ be bases of $C_*, C_*', C_*''$ and $H_*(C_*), H_*(C_*'), H_*(C_*'')$.
Picking a lift of $c_i''$ in $C_i$ and combining it with the image of $c_i'$ in $C_i$, we obtain a basis $c_i' c_i''$ of $C_i$.
We denote by $\mathcal{H}_*$ the corresponding long exact sequence in homology and by $\boldsymbol{d}$ the basis of $\mathcal{H}_*$ obtained by combining $\boldsymbol{h}, \boldsymbol{h}', \boldsymbol{h}''$.

\begin{lem}(\cite[Theorem 3.\ 1]{M1}) \label{lem_M}
If $[c_i' c_i'' / c_i] = 1$ for all $i$, then
\[ \tau(C_*, \boldsymbol{c}, \boldsymbol{h}) = \tau(C_*', \boldsymbol{c}', \boldsymbol{h}') \tau(C_*'', \boldsymbol{c}'', \boldsymbol{h}'') \tau(\mathcal{H}_*, \boldsymbol{d}). \]
\end{lem} 

The following lemma is a certain non-commutative version of \cite[Theorem 2.2]{T1}.
Turaev's proof can be easily applied to this setting.

\begin{lem} \label{lem_T}
If $C_*$ is acyclic and we find a decomposition $C_* = C_*' \oplus C_*''$ such that $C_i'$ and $C_i''$ are spanned by subbases of $c_i$ and the induced map $pr_{C_{i-1}''} \circ \partial_i |_{C_i'} \colon C_i' \to C_{i-1}''$ is an isomorphism for each $i$, then
\[ \tau(C_*, \boldsymbol{c}) = \pm \prod_{i=0}^n (\det pr_{C_{i-1}''} \circ \partial|_{C_i'})^{(-1)^i}. \]
\end{lem}

Let $\varphi \colon \Z[\pi_1 X] \to \K$ be a ring homomorphism.
We take a cell decomposition of $X$ and define the twisted homology group associated to $\varphi$ as follows:
\[ H_i^{\varphi}(X; \K) := H_i(C_*(\widetilde{X}) \otimes_{\Z[\pi_1 X]} \K), \]
where $\widetilde{X}$ is the universal covering of $X$.

\begin{defn}
If $H_*^{\varphi}(X; \K) = 0$, then we define the \textit{Reidemeister torsion} $\tau_{\varphi}(X)$ associated to $\varphi$ as follows.
We choose a lift $\tilde{e}$ in $\widetilde{X}$ for each cell $e$ of $X$. Then
\[ \tau_{\varphi}(X) := [ \tau(C_*(\widetilde{X}) \otimes_{\Z[\pi_1 X]} \K, \langle \tilde{e} \otimes 1 \rangle_e)] \in \K_{ab}^{\times} / \pm \varphi(\pi_1 X). \]
\end{defn}

We can check that $\tau_{\varphi}(X)$ does not depend on the choice of $\tilde{e}$.
It is known that Reidemeister torsion is a simple homotopy invariant of a finite CW-complex.

\subsection{The Novikov complex} \label{sec_2_2}

Next we review the Novikov complex of $f$, which is the simplest version of Novikov's construction for closed $1$-forms in \cite{N}.
See also \cite{Paj2} and \cite{Paz1}.

We can lift $f$ to a function $\tilde{f} \colon \widetilde{X} \to \R$.
If $p$ is a critical point of $f$ or $\tilde{f}$, the \textit{unstable manifold} $\mathcal{D}(p)$ is the set of all points $x$ such that the upward gradient flow starting at $x$ converges to $p$.
Similarly, the \textit{stable manifold} $\mathcal{A}(p)$ is the set of all points $x$ such that the downward gradient flow starting at $x$ converges to $p$.
Recall that we chose a Riemann metric such that $\mathcal{D}(p) \pitchfork \mathcal{A}(p)$ for any critical points $p, q$ of $f$.

We take the ``downward'' generator $t$ of $\pi_1 S^1$.

\begin{defn}
We define the \textit{Novikov completion} $\Lambda_f$ of $\Z[\pi_1 X]$ associated to $f_* \colon \pi_1 X \to \langle t \rangle$ to be the set of a formal sum $\sum_{\gamma \in \pi_1 X} a_{\gamma} \gamma$ such that $a_{\gamma} \in \Z$ and for any $k \in \Z$, the number of $\gamma$ such that $a_{\gamma} \neq 0$ and $\deg f_*(\gamma) \leq k$ is finite. 
\end{defn}

\begin{defn}
The \textit{Novikov complex} $(C_*^{Nov}(f), \partial_*^f)$ of $f$ is defined as follows.
For each critical point $p$ of $f$, we choose a lift $\tilde{p} \in \widetilde{X}$.
Then we define $C_i^{Nov}(f)$ to be the free right $\Lambda_f$-module generated by the lifts $\tilde{p}$ of index $i$.
If $p$ is a critical point of index $i$, then
\[ \partial_i^f(\tilde{p} \cdot \gamma) := \sum_{q \text{ of index } i-1, \gamma' \in \pi_1 X} n(\tilde{p} \cdot \gamma, \tilde{q} \cdot \gamma') \tilde{q} \cdot \gamma', \]
where $n(\tilde{p} \cdot \gamma, \tilde{q} \cdot \gamma')$ is the algebraic intersection number of $\mathcal{D}(\tilde{p} \cdot \gamma)$, $\mathcal{A}(\tilde{q} \cdot \gamma')$ and an appropriate level set, which can be seen as the signed number of negative gradient flow lines from $\tilde{p} \cdot \gamma$ to $\tilde{q} \cdot \gamma'$.
By the linear extension we obtain the differential $\partial_i^f \colon C_i^{Nov}(f) \to C_{i-1}^{Nov}(f)$.
\end{defn}

Obviously, the definition dose not depend on the choices of $\tilde{f}$ and $\{ \tilde{p} \}$.
It is known that appropriate orientations of stable and unstable manifolds ensure that $\partial_{i-1}^f \circ \partial_i^f = 0$.

\begin{thm}[\cite{Paz1}] \label{thm_P}
The Novikov complex $C_*^{Nov}(f)$ is chain homotopic to $C_*(\widetilde{X}) \otimes_{\Z[\pi_1 X]} \Lambda_f$.
\end{thm}

\section{The main theorem} \label{sec_3}
\subsection{Non-commutative zeta functions} \label{sec_3_1}
First we introduce a non-commutative zeta function $\zeta_f$ associated to $f$, which is closely related to the \textit{total Lefschetz-Nielsen invariant} of a self map in \cite{GN}.

Let $\Lambda_f^{\times}$ be the group of units of $\Lambda_f$.
Namely $\Lambda_f^{\times}$ consists of elements of $\Lambda_f$ having left and right inverses.
For $x, y \in (\Lambda_f^{\times})_{ab} := \Lambda_f^{\times} / [\Lambda_f^{\times}, \Lambda_f^{\times}]$, we write
\[ x \sim y \]
if for any $k \in \Z$, there exist representatives $\sum_{\gamma \in \pi_1 X} a_{\gamma} \gamma, \sum_{\gamma \in \pi_1 X} b_{\gamma} \gamma \in \Lambda_f^{\times}$ of $x, y$ respectively such that for any $\gamma \in \pi_1 X$ with $\deg f_*(\gamma) \leq k$, $a_{\gamma} = b_{\gamma}$.

\begin{lem} \label{lem_C}
The relation $\sim$ is an equivalence relation in $(\Lambda_f^{\times})_{ab}$.
\end{lem}

\begin{proof}
We only need to show the transitivity.
We assume that $x \sim y$ and $y \sim z$ for $x, y ,z \in (\Lambda_f^{\times})_{ab}$ and for any $k \in \Z$ take representatives $\sum_{\gamma} a_{\gamma} \gamma, \sum_{\gamma} b_{\gamma} \gamma$ and $\sum_{\gamma} b_{\gamma}' \gamma, \sum_{\gamma} c_{\gamma}' \gamma$ of $x, y$ and $y, z$ respectively such that for any $\gamma$ with $\deg f_*(\gamma) \leq k$, $a_{\gamma} = b_{\gamma}$ and $b_{\gamma}' = c_{\gamma}'$.
There exists $\lambda = \sum_{\gamma} d_{\gamma} \gamma \in [\Lambda_f^{\times}, \Lambda_f^{\times}]$ such that $\sum_{\gamma} b_{\gamma} \gamma = (\sum_{\gamma} b_{\gamma}' \gamma) \lambda$.
Note that for any $\gamma$ with $deg f_*(\gamma) < 0$, $d_{\gamma} = 0$.
We set $c_{\gamma}$ so that $\sum_{\gamma} c_{\gamma} \gamma = (\sum_{\gamma} c_{\gamma}' \gamma) \lambda$.
Then $\sum_{\gamma} c_{\gamma} \gamma$ is also a representative of $z$, and for any $\gamma$ with $\deg f_*(\gamma) \leq k$, $a_{\gamma} = c_{\gamma}$.
We thus get $x \sim z$.
\end{proof}

We define $\overline{(\Lambda_f^{\times})}_{ab}$ to be the quotient set by the equivalence relation.
The abelian group structure of $(\Lambda_f^{\times})_{ab}$ naturally induces that of $\overline{(\Lambda_f^{\times})}_{ab}$.

A \textit{closed orbit} is a non-constant map $o \colon S^1 \to X$ with $\frac{do}{ds} = - \nabla f$.
Two closed orbits are called equivalent if they differ by linear parameterization.
We denote by $\mathcal{O}$ the set of the equivalence classes of closed orbits.
The \textit{period} $p(o)$ is the largest integer $p$ such that $o$ factors through a $p$-fold covering $S^1 \to S^1$. 
We assume that all the closed orbits are \textit{nondegenerate}, namely the determinant of $id - d \phi \colon T_x X / T_x o(S^1) \to T_x X / T_x o(S^1)$ does not vanish  for any $[o] \in \mathcal{O}$, where $\phi$ is a $p(o)$th return map around a point $x \in o(S^1)$.
The \textit{Lefschetz sign} $\epsilon(o)$ is the sign of the determinant.
We denote by $i_+(o)$ and $i_-(o)$ the numbers of real eigenvalues  of $d \phi \colon T_x X / T_x o(S^1) \to T_x X / T_x o(S^1)$ for a return map $\phi$ which are $> 1$ and $< -1$ respectively.

\begin{defn}
We number $[o] \in \mathcal{O}$ with $p(o) = 1$ as $\{ [o_i] \}_{i=1}^{\infty}$ and choose a path $\sigma_{o_i}$ from the base point of $X$ to a point of $o_i(S^1)$ for each $[o_i]$.
Then we have $[\sigma_{o_i} o_i \bar{\sigma}_{o_i}] \in \pi_1 X$, where $\bar{\sigma}_{o_i}$ is the inverse path of $\sigma_{o_i}$.
We define
\[ \zeta_f := \left[ \prod_{i=1}^{\infty} (1 - (-1)^{i_-(o_i)} [\sigma_{o_i} o_i \bar{\sigma}_{o_i}])^{(-1)^{i_+(o_i) + i_-(o_i) + 1}} \right] \in \overline{(\Lambda_f^{\times})}_{ab}. \]
\end{defn}

By the completeness of $\Lambda_f$ we can easily check that the infinite product $\prod_{i=1}^{\infty} (1 - (-1)^{i_-(o_i)} [\sigma_{o_i} o_i \bar{\sigma}_{o_i}])^{(-1)^{i_+(o_i) + i_-(o_i) + 1}} \in \Lambda_f^{\times}$ makes sense.

\begin{lem}
The zeta function $\zeta_f$ does not depend on the choices of $\{ [o_i] \}_{i=1}^{\infty}$ and $\sigma_{o_i}$.
\end{lem}

\begin{proof}
We take another sequence $\{ [o_i'] \}_{i=1}^{\infty}$ and another path $\sigma_{o_i}'$ for each $[o_i]$.
For any $k \in \Z$, since $\{ [o_i] \in \mathcal{O} ~;~ \deg f_*([o_i]) \leq k \}$, which equals $\{ [o_i'] \in \mathcal{O} ~;~ \deg f_*([o_i']) \leq k \}$, is a finite set and
\[ [1 \pm [\sigma_{o_i} o_i \bar{\sigma}_{o_i}]] = [1 \pm [\sigma_{o_i}' o_i \bar{\sigma}_{o_i}']] \]
in $(\Lambda_f^{\times})_{ab}$, 
\begin{align*}
&\left[ \prod_{i=1}^{\infty} (1 \pm [\sigma_{o_i} o_i \bar{\sigma}_{o_i}])^{\pm 1} \right] & & \\
= &\left[ \prod_{1 \leq i \leq \infty, \deg f_*([o_i]) \leq k} (1 \pm [\sigma_{o_i} o_i \bar{\sigma}_{o_i}])^{\pm 1} \right] & \left[ \prod_{1 \leq i \leq \infty, \deg f_*([o_i]) > k} (1 \pm [\sigma_{o_i} o_i \bar{\sigma}_{o_i}])^{\pm 1} \right] & \\
= &\left[ \prod_{1 \leq i \leq \infty, \deg f_*([o_i']) \leq k} (1 \pm [\sigma_{o_i'} o_i' \bar{\sigma}_{o_i'}])^{\pm 1} \right] & \left[ \prod_{1 \leq i \leq \infty, \deg f_*([o_i]) > k} (1 \pm [\sigma_{o_i} o_i \bar{\sigma}_{o_i}])^{\pm 1} \right] & \\
= &\left[ \prod_{1 \leq i \leq \infty, \deg f_*([o_i']) \leq k} (1 \pm [\sigma_{o_i'}' o_i' \bar{\sigma}_{o_i'}'])^{\pm 1} \right] & \left[ \prod_{1 \leq i \leq \infty, \deg f_*([o_i]) > k} (1 \pm [\sigma_{o_i} o_i \bar{\sigma}_{o_i}])^{\pm 1} \right] &.
\end{align*}
in $(\Lambda_f^{\times})_{ab}$.
Therefore for any $k \in \Z$, the products $\left[ \prod_{i=1}^{\infty} (1 \pm [\sigma_{o_i} o_i \bar{\sigma}_{o_i}])^{\pm 1} \right]$ and $\left[ \prod_{i=1}^{\infty} (1 \pm [\sigma_{o_i'}' o_i' \bar{\sigma}_{o_i'}'])^{\pm 1} \right]$ have representatives in $\Lambda_f^{\times}$ such that for any $\gamma \in \pi_1 X$ with $\deg f_*(\gamma) \leq k$, the coefficients of $\gamma$ are same, and the lemma follows.
\end{proof}

By the above lemma we can write 
\[ \zeta_f = \prod_{[o] \in \mathcal{O}, p(o) = 1} [1 - (-1)^{i_-(o)} [\sigma_{o} o \bar{\sigma}_{o}]]^{(-1)^{i_+(o) + i_-(o) + 1}}. \]

Let $\Lambda_f^+$ be the subring of $\Lambda_f$ whose element $\sum_{\gamma \in \pi_1 X} a_{\gamma} \gamma \in \Lambda_f$ satisfies the fact that $a_{\gamma} = 0$ if $\deg f_*(\gamma) \leq 0$.
There is a formal exponential $\exp \colon \Lambda_f^+ \to \Lambda_f$ given by $\exp(\lambda) := \sum_{n=0}^{\infty} \frac{\lambda^n}{n!}$.
Since
\[ \epsilon(o^j) = (-1)^{i_+(o) + (j + 1) i_-(o)} \]
and
\[ \exp \left( \sum_{j=1}^{\infty} \frac{(\pm \gamma)^j}{j} \right) = (1 \mp \gamma)^{-1} \]
for $\gamma \in \pi_1 X$ with $\deg f_*(\gamma) > 0$,
\[ \zeta_f = \prod_{[o] \in \mathcal{O}, p(o) = 1} \left[ \exp \left( \sum_{j=1}^{\infty} \frac{\epsilon(o^j)}{j} [\sigma_o o^j \bar{\sigma}_o] \right) \right], \]
where $o^j$ is the composition of a $j$-fold covering $S^1 \to S^1$ and $o$.
    
\subsection{Novikov-type skew fields}
Here we proceed to construct Novikov-type non-commutative coefficients for torsion and formulate the main theorem.

A group $G$ is called \textit{poly-torsion-free-abelian (PTFA)} if there exists a filtration
\[ 1 = G_0 \triangleleft G_1 \triangleleft \dots \triangleleft G_{n-1} \triangleleft G_n = G \]
such that $G_i / G_{i-1}$ is torsion free abelian.

\begin{prop}[\cite{Pas}]
If $G$ is a PTFA group, then $\Q[G]$ is a right (and left) Ore domain; namely $\Q[G]$ embeds in its classical right ring of quotients $\Q[G](\Q[G] \setminus 0)^{-1}$.
\end{prop}

Let $G$ be a PTFA group and $\rho \colon \pi_1 X \to G$, $\alpha \colon G \to \langle t \rangle$ be group homomorphisms such that $\alpha \circ \rho = f_*$.
Then $\ker \alpha$ is also PTFA, and so we have the classical ring of quotients $\mathcal{K}$ of $\Q[\ker \alpha]$.
We denote by $l$ the nonnegative integer such that $t^l$ generates $\im \alpha$.
We pick $\mu \in G$ such that $\alpha(\mu) = t^l$ and let $\theta \colon \mathcal{K} \to \mathcal{K}$ be the automorphism given by $\theta(k) = \mu k \mu^{-1}$ for $k \in \mathcal{K}$.
Now we have a Novikov type skew field $\mathcal{K}_{\theta}((t^l))$.
More precisely, the elements of $\mathcal{K}_{\theta}((t^l))$ are formal sums $\sum_{i = n}^{\infty} a_i t^{l i}$ with $n \in \Z$ and $a_i \in \mathcal{K}$, and the multiplication is defined by using the rule $t^l k = \theta(k) t^l$.
Note that the isomorphism type of the ring $\mathcal{K}_{\theta}((t^l))$ dose not depend on the choice of $\mu$, and we can regard $\Z[G]$ as a subring of $\mathcal{K}_{\theta}((t^l))$.

For $x, y \in \mathcal{K}_{\theta}((t))_{ab}^{\times}$, we write
\[ x \sim y \]
if for any $k \in \Z$, there exist representatives $\sum_{i \in \Z} a_i t^{l i}, \sum_{i \in \Z} b_i t^{l i} \in \mathcal{K}_{\theta}((t^l))^{\times}$ of $x, y$ respectively such that for any $i \leq k$, $a_i = b_i$.

The following lemma can be similarly proved as Lemma \ref{lem_C} and so we omit the proof.

\begin{lem}
The relation $\sim$ is an equivalence relation in $\mathcal{K}_{\theta}((t^l))_{ab}^{\times}$.
\end{lem}

We define $\overline{\mathcal{K}_{\theta}((t^l))}_{ab}^{\times}$ to be the quotient set by the equivalence relation, which is also an abelian group.
Note that if $\mathcal{K}_{\theta}((t^l))$ is commutative, then $\overline{\mathcal{K}_{\theta}((t^l))}_{ab}^{\times} = \mathcal{K}_{\theta}((t^l))_{ab}^{\times}$.
Moreover We can show that the natural map $\mathcal{K}_{ab}^{\times} \to \overline{\mathcal{K}_{\theta}((t^l))}_{ab}^{\times}$ is injective as follows.
Let $a \in \mathcal{K}^{\times}$ and assume $[a] = 1 \in \overline{\mathcal{K}_{\theta}((t^l))}_{ab}^{\times}$.
Then there exists $\lambda \in [\mathcal{K}_{\theta}((t^l))^{\times}, \mathcal{K}_{\theta}((t^l))^{\times}]$ such that $a$ is equal to the degree $0$ part of $\lambda$, which is in $[\mathcal{K}^{\times}, \mathcal{K}^{\times}]$.
Therefore $[a] = 1 \in \mathcal{K}_{ab}^{\times}$.

The group homomorphism $\rho$ naturally extends to a ring homomorphism $\Lambda_f \to \mathcal{K}_{\theta}((t^l))$.
By abuse of notation, we also denote it by $\rho$.
By virtue of Theorem \ref{thm_P} $H_*^{\rho}(X ; \mathcal{K}_{\theta}((t^l)))$ is isomorphic to $H_*(C_*^{Nov}(f) \otimes_{\Lambda} \mathcal{K}_{\theta}((t^l)))$.

\begin{defn}
If $H_*^{\rho}(X ; \mathcal{K}_{\theta}((t^l))) = 0$, then we define the \textit{Novikov torsion} $\tau_{\varphi}(X)$ associated to $\rho$ as
\[ \tau_{\rho}^{Nov}(f) := [\tau(C_*^{Nov}(f) \otimes_{\Lambda_f} \mathcal{K}_{\theta}((t^l)), \langle \tilde{p} \otimes 1 \rangle_p)] \in \mathcal{K}_{\theta}((t^l))_{ab}^{\times} / \pm \rho(\pi_1 X). \]
\end{defn}

The ring homomorphism $\rho \colon \Lambda_f \to \mathcal{K}_{\theta}((t^l))$ naturally induces a group homomorphism $\rho_* \colon \overline{(\Lambda_f^{\times})}_{ab} \to \overline{\mathcal{K}_{\theta}((t^l))}_{ab}^{\times}$ and $\overline{\mathcal{K}_{\theta}((t^l))}_{ab}^{\times} / \pm \rho(\pi_1 X)$ is a quotient group of $\mathcal{K}_{\theta}((t^l))_{ab}^{\times} / \pm \rho(\pi_1 X)$.

\begin{thm}[Main theorem] \label{thm_M}
For a given pair $(\rho, \alpha)$ as above, if $H_*^{\rho}(X ; \mathcal{K}_{\theta}((t^l))) = 0$, then
\[ \tau_{\rho}(X) = \rho_*(\zeta_f) \tau_{\rho}^{Nov}(f) \in \overline{\mathcal{K}_{\theta}((t^l))}_{ab}^{\times} / \pm \rho(\pi_1 X). \]
\end{thm}

\begin{rem}
More generally, the similar construction makes sense and the theorem also holds under the assumption that $\Q[\ker \alpha]$ is a right Ore domain instead of that $G$ is PTFA.
Moreover it is expected that we can eliminate the ambiguity of multiplication by an element of $\rho(\pi_1 X)$, carefully considering \textit{Euler structures} by Turaev \cite{T1}, \cite{T2} as in \cite{HL2}. 
\end{rem}

An important example of a pair $(\rho, \alpha)$ is provided by Harvey's \textit{rational derived series} in \cite{Ha}.

\begin{defn}
For a group $\Pi$, let $\Pi_r^{(0)} = \Pi$ and we inductively define
\[ \Pi_r^{(i)} = \{ \gamma \in \Pi_r^{(i-1)} ~;~ \gamma^k \in [\Pi_r^{(i-1)}, \Pi_r^{(i-1)}] \text{ for some } k \in \Z \setminus 0 \}. \]
\end{defn}

\begin{lem}[\cite{Ha}]
For any group $\Pi$ and any $n$,
\[ \Pi_r^{(n-1)} / \Pi_r^{(n)} = (\Pi_r^{(n-1)} / [\Pi_r^{(n-1)}, \Pi_r^{(n-1)}]) / \text{ torsion} \]
and $\Pi / \Pi_r^{(n)}$ is a PTFA group.
\end{lem}

For any $n$, we have the natural surjection $\rho^{(n)} \colon \pi_1 X \to \pi_1 X / (\pi_1 X)_r^{(n+1)}$ and the induced homomorphism $\alpha^{(n)} \colon \pi_1 X / (\pi_1 X)_r^{(n+1)} \to \langle t \rangle$ by $f_*$.
By the above construction we obtain the group homomorphism $\rho^{(n)} \colon \Z[\pi_1 X] \to \mathcal{K}_{\theta}^{(n)}(t^l)$ and the extended one $\tilde{\rho}^{(n)} \colon \Lambda_f \to \mathcal{K}_{\theta}^{(n)}((t^l))$, where $\mathcal{K}_{\theta}^{(n)}(t^l)$ is the subfield of $\mathcal{K}_{\theta}^{(n)}((t^l))$ consisting of rational elements. 

\begin{defn}
If $H_*^{\rho^{(n)}}(X ; \mathcal{K}_{\theta}^{(n)}(t^l)) = 0$, then $\tau_{\rho^{(n)}}(X) \in \mathcal{K}_{\theta}^{(n)}(t^l)_{ab}^{\times} / \pm \rho^{(n)}(\pi_1 X)$ is defined.
We call it the \textit{higher-order Reidemeister torsion} of order $n$.
\end{defn}

\begin{rem}
It is known by Friedl that $\tau_{\rho^{(n)}}(X)$ equals an alternating product of non-commutative Alexander polynomials.
See \cite{F} for the details.
\end{rem}

If $H_*^{\rho^{(n)}}(X ; \mathcal{K}_{\theta}^{(n)}(t^l)) = 0$, then we have $H_*^{\tilde{\rho}^{(n)}}(X ; \mathcal{K}_{\theta}^{(n)}((t^l))) = 0$, and we can also define $\tau_{\tilde{\rho}^{(n)}}(X) \in \mathcal{K}_{\theta}^{(n)}((t^l))_{ab}^{\times} / \pm \rho^{(n)}(\pi_1 X)$.
By the functoriality of Reidemeister torsion the image of $\tau_{\rho^{(n)}}(X)$ by the natural map $\mathcal{K}_{\theta}^{(n)}(t^l)_{ab}^{\times} / \pm \rho^{(n)}(\pi_1 X) \to \mathcal{K}_{\theta}^{(n)}((t^l))_{ab}^{\times} / \pm \rho^{(n)}(\pi_1 X)$ equals $\tau_{\tilde{\rho}^{(n)}}(X)$.
Thus for the pair $(\rho^{(n)}, \alpha^{(n)})$, the main theorem gives a Morse theoretical and dynamical presentation of $\tau_{\rho^{(n)}}(X)$ in $\overline{\mathcal{K}_{\theta}^{(n)}((t^l))}_{ab}^{\times} / \pm \rho^{(n)}(\pi_1 X)$ as a corollary.

\section{Proof} \label{sec_4}

The proof of the main theorem is divided into two parts.
In the first part we construct an ``approximate'' CW complex $X'$ which is adapted to $\nabla f$, and we show that the Reidemeister torsion of $X'$ equals that of $X$.
The second part is devoted to computation of the torsion of $X'$, and we see that it has the desired form.

\subsection{An approximate CW-complex} \label{subsec_4_1}
Let $\Sigma$ be a level set of a regular value of $f$ and let $Y$ be the compact Riemannian manifold obtained by cutting $X$ along $\Sigma$.
We can pick a Morse function $f_0 \colon Y \to \R$ induced by $f$.
We write $\partial Y = \Sigma_0 \sqcup \Sigma_1$, where $\Sigma_0, \Sigma_1$ are the cutting hypersurfaces and $- \nabla f_0$ points outward along $\Sigma_0$.
We denote by $\mathcal{A}_0(p), \mathcal{D}_0(p)$ the stable and unstable manifolds of a critical point $p$ of $f_0$.

We take a smooth triangulation $T_1$ of $\Sigma_1$ such that each simplex is transverse to $\mathcal{A}_0(p)$ for each critical point $p$ of $f_0$.
For $\sigma \in T_1$, let us denote by $\mathcal{F}(\sigma)$ the set of all $y \in Y$ such that the flow of $\nabla f_0$ starting at $y$ hits $\sigma$.
It is well-known that the submanifolds $\mathcal{D}_0(p)$ and $\mathcal{F}(\sigma)$ have natural compactifications $\overline{\mathcal{D}_0(p)}$ and $\overline{\mathcal{F}(\sigma)}$ respectively by adding broken flow lines of $-\nabla f_0$.
(See for instance \cite{HL1}.) 
We choose a cell decomposition $T_0$ of $\Sigma_0$ such that $\overline{\mathcal{D}_0(p)} \cap \Sigma_0$ and $\overline{\mathcal{F}(\sigma)} \cap \Sigma_0$ are subcomplexes for each critical point $p$ and each simplex $\sigma$.
Then we can check that the cells in $T_0, T_1$, $\overline{\mathcal{D}_0(p)}$ and $\overline{\mathcal{F}(\sigma)}$ give a cell decomposition $T_Y$ of $Y$.

Let $h \colon (\Sigma_0, T_0) \to (\Sigma_1, T_1)$ be a cellular approximation to the canonical identification $\Sigma_0 \to \Sigma_1$. 
We consider the mapping cylinder $M_h$ of $h$:
\[ M_h := ((\Sigma_0 \times [0, 1]) \sqcup \Sigma_1) / (x, 1) \sim h(x). \]
It has a natural cell decomposition induced by $T_0$ and $T_1$.

\begin{defn}
Let $X'$ be the space obtained by gluing $Y$ and $M_h$ along $\Sigma_0 \sqcup \Sigma_1$.
\end{defn}

For a cell $\Delta$ in $T_Y$ of the form $\overline{\mathcal{D}_0(p)}$ and $\overline{\mathcal{F}(\sigma)}$, we define $\widehat{\Delta}$ to be the set obtained by gluing $\Delta$ and
\[ (((\Delta \cap \Sigma_0) \times [0, 1]) \sqcup h(\Delta \cap \Sigma_0)) / (x, 1) \sim h(x) \]
along $\Delta \cap \Sigma_0$.
Cells of the form $\widehat{\mathcal{D}_0(p)}$, $\sigma$ and $\widehat{\mathcal{F}(\sigma)}$ for a critical point $p$ of $f_0$ and $\sigma \in T_1$ give a cell decomposition of $X'$.

We pick a homotopy equivalent map $X' \to X$ and identify $\pi_1 X'$ with $\pi_1 X$.

\begin{lem} \label{lem_p1}
Under the assumptions of Theorem \ref{thm_M} we have
\[ \tau_{\rho}(X) = \tau_{\rho}(X'). \]
\end{lem}

\begin{proof}
In all of the calculations below, we implicitly tensor the chain complexes with the skew field $\mathcal{K}_{\theta}((t^l))$, and brackets mean that they are in $\mathcal{K}_{\theta}((t^l))_{ab}^{\times} / \pm \rho(\pi_1 X)$.

We regard $X$ as the union of $Y$ and $\Sigma \times [0, 1]$ along $\Sigma_0 \sqcup \Sigma_1$, then we have short exact sequences
\begin{align*}
&0 \to C_*(\widetilde{\Sigma}_0) \oplus C_*(\widetilde{\Sigma}_1) \to C_*(\widetilde{\Sigma} \times [0, 1]) \oplus C_*(\widetilde{Y}) \to C_*(\widetilde{X}) \to 0, \\
&0 \to C_*(\widetilde{\Sigma}_0) \oplus C_*(\widetilde{\Sigma}_1) \to C_*(\widetilde{M_{h}}) \oplus C_*(\widetilde{Y}) \to C_*(\widetilde{X}) \to 0.
\end{align*}

The natural surjection $\Sigma \times [0,1] \to M_h$ induces an isomorphism between $H_*^{\rho}(\Sigma \times [0,1]; \mathcal{K}_{\theta}((t^l)))$ and $H_*^{\rho}(M_h; \mathcal{K}_{\theta}((t^l)))$, and there is an isomorphism between the long exact sequences in homology for the above sequences.
Let $\boldsymbol{c}$ and $\boldsymbol{c}'$ be the bases of $C_*(\widetilde{\Sigma} \times [0, 1])$ and $C_*(\widetilde{M_{h}})$ which are obtained by $T_0$ and the product cell structure from $T_1$.
We pick bases $\boldsymbol{h}$ and $\boldsymbol{h}'$ of $H_*^{\rho}(\Sigma \times [0,1]; \mathcal{K}_{\theta}((t^l)))$ and $H_*^{\rho}(M_h; \mathcal{K}_{\theta}((t^l)))$ such that the isomorphism maps $\boldsymbol{h}$ to $\boldsymbol{h}'$.
Then from Lemma \ref{lem_M} we obtain
\begin{equation} \label{eq_p1}
\frac{\tau_{\rho}(X)}{\tau_{\rho}(X')} = \frac{[\tau(C_*(\widetilde{\Sigma} \times [0, 1]), \boldsymbol{c}, \boldsymbol{h})]}{[\tau(C_*(\widetilde{M}_h), \boldsymbol{c}', \boldsymbol{h}')]}.
\end{equation}

We have short exact sequences
\begin{align*}
&0 \to C_*(\widetilde{\Sigma} \times 1) \to C_*(\widetilde{\Sigma} \times [0, 1]) \to C_*(\widetilde{\Sigma} \times [0, 1], \widetilde{\Sigma} \times 1) \to 0, \\
&0 \to C_*(\widetilde{\Sigma}_1) \to C_*(\widetilde{M}_h) \to C_*(\widetilde{M}_h, \widetilde{\Sigma}_1) \to 0.
\end{align*}
The surjection $\Sigma \times [0,1] \to M_h$ also induces an isomorphism between the long exact sequences in homology for the above sequences.
Let $\boldsymbol{d}$ and $\boldsymbol{d}'$ be bases of $C_*(\widetilde{\Sigma} \times [0, 1], \widetilde{\Sigma} \times 1)$ and $C_*(\widetilde{M}_h, \widetilde{\Sigma}_1)$ induced by $\boldsymbol{c}$ and $\boldsymbol{c}'$.
Then again from Lemma \ref{lem_M} we obtain
\begin{equation} \label{eq_p2}
\frac{[\tau(C_*(\widetilde{\Sigma} \times [0, 1]), \boldsymbol{c}, \boldsymbol{h})]}{[\tau(C_*(\widetilde{M}_h), \boldsymbol{c}', \boldsymbol{h}')]} = \frac{[\tau(C_*(\widetilde{\Sigma} \times [0, 1], \widetilde{\Sigma} \times 1), \boldsymbol{d})]}{[\tau(C_*(\widetilde{M}_h, \widetilde{\Sigma}_1), \boldsymbol{d}')]}.
\end{equation}

By direct computations we have
\[ [\tau(C_*(\widetilde{\Sigma} \times [0, 1], \widetilde{\Sigma} \times 1), \boldsymbol{d})] = [\tau(C_*(\widetilde{M}_h, \widetilde{\Sigma}_1), \boldsymbol{d}')] = [1]. \]
Now the lemma follows from \eqref{eq_p1}, \eqref{eq_p2} and these equalities.
\end{proof}

\subsection{Computation of the torsion}  \label{subsec_4_2}

We decompose
\[ C_i(\widetilde{X}') \otimes_{\Z[\pi_1 X']} \mathcal{K}_{\theta}((t^l)) = D_i \oplus E_i \oplus F_i,  \]
where $D_i$, $E_i$ and $F_i$ are generated by elements of the form $\widehat{\mathcal{D}_0(p)}$, $\sigma$ and $\widehat{\mathcal{F}(\sigma)}$ for a critical point $p$ of $f_0$ and $\sigma \in T_1$ respectively.
There are natural identifications $D_i \cong C_i^{Nov}(f) \otimes_{\Lambda_f} \mathcal{K}_{\theta}((t^l))$ and $F_i \cong E_{i-1}$.
Then the matrix for the differential $\partial_i$ can be written as

\begin{align*}
& \quad D_i \hspace{16pt} E_i \hspace{24pt} F_i \\
\partial_i ~=~ 
\begin{matrix}
D_{i-1} \\
E_{i-1} \\
F_{i-1}
\end{matrix}
\quad &
\begin{pmatrix}
\boldsymbol{N}_i & 0 & \boldsymbol{W}_i \\
- \boldsymbol{M}_i & \partial_i^{\Sigma} & I - \phi_{i-1} \\
0 & 0 & - \partial_i^{\Sigma}
\end{pmatrix},
\end{align*}
where $\partial_i^{\Sigma}$ is the differential on $C_*(\widetilde{\Sigma}) \otimes_{\Z[\pi_1 \Sigma]} \mathcal{K}_{\theta}((t^l))$ and $\phi_{i-1}$ can be interpreted as the return map of the gradient flow in $\widetilde{X}$ after perturbation by $h$.
We set
\[ \boldsymbol{K}_i := \boldsymbol{N}_i + \boldsymbol{W}_i (I - \phi_{i-1})^{-1} \boldsymbol{M}_i \colon D_i \to D_{i-1}. \]

Since $C_*(\widetilde{X}) \otimes_{\Z[\pi_1 X]} \mathcal{K}_{\theta}((t^l))$ is acyclic, the Novikov complex $(D_*, \partial_*^f)$ is also acyclic by Theorem \ref{thm_P}, and we can choose a decomposition $D_i = D_i' \oplus D_i''$ such that $D_i'$ and $D_i''$ are spanned by lifts of the critical points of $f$ and $\partial_f$ induces an isomorphism $D_i' \to D_{i-1}''$.
We denote by $K_i \colon D_i' \to D_{i-1}''$ the induced map by $\boldsymbol{K}_i$.

\begin{lem} \label{lem_p2}
Under the assumptions of Theorem \ref{thm_M}, if $K_i$ is non-singular for each $i$, then
\[ \tau_{\rho}(X') =  \prod_{i=1}^d [\det(I - \phi_{i-1}) \det K_i]^{(-1)^i}. \]
\end{lem}

\begin{proof}
We consider the matrix
\begin{align*}
& \quad D_i' \hspace{27pt} F_i \\
\Omega_i ~:=~
\begin{matrix}
D_{i-1}'' \\
E_{i-1}
\end{matrix}
\quad &
\begin{pmatrix}
N_i & W_i \\
- M_i & I - \phi_{i-1}
\end{pmatrix},
\end{align*}
where $M_i \colon D_i' \to E_{i-1}$, $N_i \colon D_i' \to D_{i-1}''$ and $W_i \colon F_i \to D_{i-1}''$ be the induced maps by $\boldsymbol{M}_i$, $\boldsymbol{N}_i$ and $\boldsymbol{W}_i$ respectively.
After elementary row operations we can turn $\Omega_i$ into the matrix
\[
\begin{pmatrix}
K_i & 0 \\
- M_i & I - \phi_{i-1}
\end{pmatrix}.
\]
Since $K_i$ is nonsingular, $\Omega_i$ is also nonsingular and
\[ \det \Omega_i = \det(I - \phi_{i-1}) \det K_i. \]
By Lemma \ref{lem_T} we have
\[
\tau_{\rho}(X') = \prod_{i=1}^{d} [\det \Omega_i]^{(-1)^i},
\]
which proves the lemma.
\end{proof}

For a positive integer $k$ and $x, y \in \overline{\mathcal{K}_{\theta}((t^l))}_{ab}^{\times} / \pm \rho(\pi_1 X)$, we write
\[ x \sim_k y \]
if there exist representatives $\sum_{i=0}^{\infty} a_i t^{l i}, \sum_{i=0}^{\infty} b_i t^{l i} \in \mathcal{K}_{\theta}((t))^{\times}$ of $x, y$ respectively such that $a_0 b_0 \neq 0$ and $a_i = b_i$ for $i = 0, 1, \dots, k$.
Note that $x = y$ if and only if for any positive integer $k$, $x \sim_k y$.

\begin{lem} \label{lem_p3}
For any positive integer $k$, if we choose $T_1$ sufficiently fine and $h$ sufficiently close to the identity, then 
\[ \prod_{i=1}^d [\det(I - \phi_{i-1})]^{(-1)^i} \sim_k [\rho_*(\zeta_f)]. \]
\end{lem}

We prepare some notation and a lemma for the proof.

Let $\varphi \colon \Sigma \setminus \sqcup_{p} \mathcal{A}_0(p) \to \Sigma \setminus \sqcup_{p} \mathcal{D}_0(p)$ be the diffeomorphism defined by the downward gradient flows and let $H \colon \Sigma \times [0, 1] \to \Sigma$ be the homotopy from $id$ to $h$.
We can consider the $i$ times iterate maps $\varphi^i$ and $(H(\cdot, t) \circ \varphi)^i$ for $t \in [0,1]$ which are partially defined.
A natural compactification $\overline{\Gamma}_t^i$ of the graph $\Gamma_t^i \in \Sigma \times \Sigma$ of $(H(\cdot, t) \circ \varphi)^i$ is defined by attaching pairs $(x, H(y, t))$, where $x$ is the starting point and $y$ is the end point of a broken flow line of $- \nabla f_0$.
(See \cite{HL1}, \cite{HL2} for more details.)

It is known that there exists a positive integer $N$ such that if the simplexes in $T_1$ are all contained in balls of radius $\epsilon$, then $H$ can be chosen so that the distance between $x$ and $H(x, t)$ is $< N \epsilon$ for all $x \in \Sigma$ and $t \in [0, 1]$.
(See for instance \cite{Sp}.)
Since the set of fixed points of $\varphi^i$ lies in the interior of $\overline{\Gamma}_0^i$ under the diagonal map $\Sigma \to \Sigma \times \Sigma$ and is compact from the nondegenerate assumption, it follows for any positive integer $k$ that we can choose $\epsilon$ so that $\overline{\Gamma}_t^i$ does not cross the diagonal in $\Sigma \times \Sigma$ for all $i \leq k$.

\begin{lem} \label{lem_p4}
Let $k$ be a positive integer and suppose that $l = 1$.
Let $(a_{i, j})$ be the $n$-dimensional matrix over $\mathcal{K}_{\theta}((t))$ such that $a_{i, j} = c_{i, j} t$, where $c_{i, j} \in \mathcal{K}$.
If $a_{1, i_1} a_{i_1, i_2} \dots a_{i_{j-1}, 1} = 0$ for any sequence $i_1, i_2, \dots, i_{j-1}$ with $j \leq k$, then
\[ [\det(\delta_{i, j} - a_{i, j})_{1 \leq i, j \leq n}] \sim_k [\det(\delta_{i, j} - a_{i, j})_{2 \leq i, j \leq n}], \]
where $\delta_{i, j}$ is Kronecker's delta.
\end{lem}

\begin{proof}
We set
\[ b_{i, j}^{(0)} = \delta_{i, j} - a_{i, j} \]
and inductively define $b_{i, j}^{(m)}$ for $m = 1, \dots, n$ as follows:
\[ b_{i, j}^{(m)} = b_{i, j}^{(m-1)} - b_{i, m}^{(m-1)} (b_{m, m}^{(m-1)})^{-1} b_{m, j}^{(m-1)}. \]
This is an elementary row operation with respect to the $i$th row and $b_{i, j}^{(m)} = 0$ if $i \neq j \leq m$.
We have
\[ [\det (b_{i, j}^{(0)})_{1 \leq i, j \leq n}] = [\det (b_{i, j}^{(n)})_{1 \leq i, j \leq n}] = \left[ \prod_{i=1}^n b_{i, i}^{(n)} \right]. \]

By induction on $m$ we first show the following observation concerning any nonzero term in $b_{i, j}^{(m)} - \delta_{i, j}$: \\
(i) The term has positive degree. \\
(ii) The term has elements $a_{i, i_1}$, $a_{i_1, i_2}$, \dots, $a_{i_{j'-1}, j}$ as factors for a sequence $i_1, i_2, \dots, i_{j'-1}$. \\
(iii) If the term has $a_{i', 1}$ as a factor for some $i'$, then then the degree of the term is $> k$ or we can make such a sequence in (ii) contains $1$.

It is easy to check them for $m = 0$.
We assume them for $m = m'-1$.
Since
\[ (b_{m', m'}^{(m'-1)})^{-1} = 1 + \sum_{i=1}^{\infty} (-1)^i (b_{m', m'}^{(m'-1)} - 1)^i, \]
(i) for $m = m'$ follows from (i) for $m = m'-1$.
By (ii) for $m = m'-1$ we see at once (ii) for $m = m'$. 
We take any nonzero term $c$ in $(b_{m', m'}^{(m'-1)})^{-1}$ which has $a_{i', 1}$ as a factor.
Then there is a nonzero term $c'$ in $b_{m', m'}^{(m'-1)}$ which has $a_{i', 1}$ as a factor such that $c$ has $c'$ as a factor.
If $c'$ has elements $a_{m', i_1}$, $a_{i_1, i_2}$, \dots, $a_{i_{j'-1}, m'}$ as factors for a sequence $i_1, i_2, \dots, i_{j'-1}$ containing $1$, then $j' \geq k$ by the assumption, and $\deg c' > k$.
Hence by (ii) and (iii) for $m = m'-1$, $\deg c \geq \deg c > k$.  
Now we can immediately check (iii) for $m = m'$.

As a consequence of the above argument the degree of any nonzero term in $b_{i, i}^{(n)}$ which has $a_{i', 1}$ as a factor is $> k$, and $\left[ \prod_{i=1}^n b_{i, i}^{(n)} \right]$ is invariant even if we erase such terms.
Therefore in considering the equivalence class we can regard $a_{i, 1}$ as $0$ for all $i$, which deduce the lemma.
\end{proof}

\begin{proof}[Proof of Lemma \ref{lem_p3}]
We only consider the case where $l = 1$.
If $l = 0$, then $\phi_i = 0$ for all $i$ and there is no closed orbit, and so there is nothing to prove.
If $l > 1$, then we can prove it by a similar argument.

We set
\[ \mathcal{I}_k := \{ [o] \in \mathcal{O} ~;~ p(o) = 1, \deg f_*([o]) \leq k \}. \]
Then we see that
\begin{equation} \label{eq_p3}
[\rho_*(\zeta_f)] \sim_k \prod_{[o] \in \mathcal{I}_k} [1 - (-1)^{i_-(o)} \rho([\sigma_o o \bar{\sigma}_o])]^{(-1)^{i_+(o) + i_-(o) + 1}}.
\end{equation}

For $[o] \in \mathcal{I}_k$, there is a sequence $x_0, x_1, \dots, x_{\deg f([o])-1}$ of fixed points of $\varphi^{\deg f_*([o])}$ such that $\varphi(x_{j-1}) = x_j$ for $j = 1, 2, \dots, \deg f([o])-1$.
Choosing $T_1$ sufficiently fine and $H$ as above, we can pick mutually disjoint contractible subcomplexes $N_{x_j}$ of $T_1$ satisfying the following conditions: \\
(i) Each $N_{x_j}$ contains all the fixed points of $(H(\cdot, t) \circ \varphi)^{\deg f_*([o])}$ which are close to $x_j$ for all $t \in [0, 1]$. \\
(ii) There is a one-to-one correspondence between cells of $N_{x_{j-1}}$ and ones of $N_{x_j}$ by $h \circ \varphi$ and $h \circ \varphi(N_{\deg f_*([o]) - 1}) \cap N_j = \emptyset$ for $j = 1, 2, \dots, \deg f_*([o]) - 1$. \\
(iii) If there is a sequence $\sigma = \sigma_0, \sigma_1, \dots, \sigma_i = \sigma \in T_1$ so that $\sigma_j \subset h \circ \varphi(\sigma_{j-1})$ for a simplex $\sigma$ not contained in any $N_{x_j}$, then $i > k$.

We denote by $N_{[o]}$ and $N_{[o]}'$ for $[o] \in \mathcal{I}_k$ the union of all $N_{x_j}$ for a fixed point $x_j \in o(S^1)$ and that of all $N_{x_0}$ for such a sequence of $[o]$.
Then we define
\begin{align*}
\phi_{[o], i} &\colon C_i(\widetilde{N}_{[o]}) \otimes \mathcal{K}_{\theta}((t)) \to C_i(\widetilde{N}_{[o]}) \otimes \mathcal{K}_{\theta}((t)) \\
\phi_{[o], i}' &\colon C_i(\widetilde{N}_{[o]}') \otimes \mathcal{K}_{\theta}((t)) \to C_i(\widetilde{N}_{[o]}') \otimes \mathcal{K}_{\theta}((t))
\end{align*}
to be the maps induced by $h \circ \varphi$ and $(h \circ \varphi)^{\deg f_*([o])}$ respectively.

Note that all the entries of $\phi_i$ are monomials.
By condition (iii) we can apply Lemma \ref{lem_p4} repeatedly and obtain
\begin{equation} \label{eq_p4}
[\det(I - \phi_i)] \sim_k \prod_{[o] \in \mathcal{I}_k} [\det(I - \phi_{[o], i})].
\end{equation}

By condition (ii) we can take simplexes $\sigma_j \subset N_{x_j}$ such that $h \circ \varphi(\sigma_{j-1}) = \sigma_j$ for $j = 1, 2, \dots, \deg f([o])-1$.
The matrix of the restriction of $I - \phi_{[o], i}$ on these simplexes has the form
\[
\begin{pmatrix}
1 & 0 & \dots & 0 & m \rho(\gamma_{\deg f_*([o])}) \\
\pm \rho(\gamma_1) & 1 & \dots & 0 & 0 \\
0 & \pm \rho(\gamma_2) & \ddots & \vdots & \vdots \\
\vdots & \vdots & \ddots & 1 & 0 \\
0 & 0 & \dots & \pm \rho(\gamma_{\deg f_*([o]) - 1}) & 1
\end{pmatrix},
\]
where $m \in \Z$ and $\gamma_j \in \pi_1 X$.
Since $\prod_{j = 0}^{\deg f_*([o]) - 1} \gamma_{\deg f_*([o]) - j} = [\sigma_o o \bar{\sigma}_o]$ for a path $\sigma_o$ from the base point of $X$ to $x_0$, the determinant of the matrix is $(1 - (\pm m \rho([\sigma_o o \bar{\sigma}_o]))$ and the coefficient of $\sigma_0$ of $\phi_{[o], i}(\sigma_0 \otimes 1)$ is $\pm m \rho([\sigma_o o \bar{\sigma}_o])$.
According to the above argument, we have
\begin{equation} \label{eq_p5}
[\det(I - \phi_{[o], i})] = [\det(I - \phi_{[o], i}')].
\end{equation}

Since the entries of the matrix $\phi_{[o], i}'$ are all in an abelian ring $\rho(\Z[[\sigma_o o \bar{\sigma}_o]])$,
\begin{align*}
\prod_{i=0}^{d-1} [\det(I - \phi_{[o], i}')] &= \left[ \exp \sum_{j=0}^{\infty} \sum_{i=0}^{d-1}  \frac{(-1)^i}{j} \tr (\phi_{[o], i}')^j \right] \\
&\sim_k \left[ \exp \sum_{j=1}^{\infty} \frac{\mathrm{Fix}((\varphi |_{N_{[o]}'})^{j \deg f_*([o])})}{j} \rho([\sigma_o o \bar{\sigma}_o])^j \right] \\
&= \left[ \exp \sum_{j=1}^{\infty} \frac{\epsilon(o^j)}{j} \rho([\sigma_o o \bar{\sigma}_o])^j \right] \\
&= [1 - (-1)^{i_-(o)} \rho([\sigma_o o \bar{\sigma}_o])]^{(-1)^{i_+(o) + i_-(o) + 1}},
\end{align*}
where $\mathrm{Fix}((\varphi |_{N_{[o]}'})^{j \deg f_*([o])})$ counts fixed points of $(\varphi |_{N_{[o]}'})^{j \deg f_*([o])}$ with sign.
The second equivalence follows from the machinery used to prove the Lefschetz fixed point theorem.
From \eqref{eq_p3}, \eqref{eq_p4}, \eqref{eq_p5} and this the lemma is proved.
\end{proof}

\begin{lem} \label{lem_p5}
For any positive integer $k$, if we choose $T_1$ sufficiently fine and $h$ sufficiently close to the identity, then $K_i$ is non-singular and 
\[ \prod_{i=1}^d [\det K_i]^{(-1)^i} \sim_k [\tau_{\rho}^{Nov}(f)]. \]
\end{lem}

\begin{proof}
Suppose to begin that $h = id$.
The unstable manifold $\mathcal{D}(p)$ of a critical point $p$ of $f$ has a natural compactification $\overline{\mathcal{D}(p)}$ such as $\mathcal{D}_0(p)$.
The compactification $\overline{\mathcal{D}(p)}$ can be represented as
\[ \overline{\mathcal{D}_0(p)} + \sum_{j=0}^{\infty} \overline{\mathcal{F}(\phi_{i-1}^j \boldsymbol{M}_i(\overline{\mathcal{D}_0(p))}}, \]
where by abuse of notation we also denote by $\mathcal{F}$ the linear extension of $\mathcal{F}$.
So if we identify $D_i$ with $C_i^{Nov}(f) \otimes_{\Lambda_f} \mathcal{K}_{\theta}((t^l))$, then
\begin{align*}
\partial_i^f(\overline{\mathcal{D}_0(p)}) &= pr_{D_{i-1}} \circ \partial_i \left( \overline{\mathcal{D}_0(p)} +\sum_{j=0}^{\infty} \overline{\mathcal{F}(\phi_{i-1}^j \boldsymbol{M}_i(\overline{\mathcal{D}_0(p)}))} \right) \\
&= \boldsymbol{K}_i(\overline{D_0(p)})
\end{align*}
for a critical point $p$ with index $i$.
Hence $\partial_i^f$ induces $K_i \colon D_i' \to D_{i-1}''$, and $K_i$ is nonsingular.
From Lemma \ref{lem_T} we have
\begin{align*}
\prod_{i=1}^d [\det K_i]^{(-1)^i} &= \prod_{i=1}^d [\det pr_{D_i''} \circ \partial_i^f |_{D_i'}]^{(-1)^i} \\
&= [\tau_{\rho}^{Nov}(f)].
\end{align*}

Next we consider the case where $h \neq id$.

Let $pr_1^j, pr_2^j \colon \overline{\Gamma}_t^j \to \Sigma$ be the restriction of the first and second projections of $\Sigma \times \Sigma$.
We define
\[ B_t^j(p) := pr_2(pr_1^{-1}(H(\cdot, t)(\overline{\mathcal{D}_0(p)} \cap \Sigma_0))) \]
for $j = 0, 1, \dots, k-1$, $t \in [0,1]$ and a critical point $p$ of $f$.
Since the set of the intersection points of $B_0^j(p)$ and $\mathcal{A}_0(q) \cap \Sigma_1$ for any critical point $q$ lies the interior of $B_0^j(p)$ and is compact from the transverse condition, it follows that we can choose $T_1$ sufficiently fine and $H$ as above so that $B_t^j(p)$ does not cross $\mathcal{A}(q) \cap \Sigma_1$ for all $j < k$, where we naturally identify $\Sigma_1$ with $\Sigma$.

By a similar argument to that of Lemma \ref{lem_p3} we can check that the image of the hat of $\mathcal{F}(\phi_{i-1}^j \boldsymbol{M}_i (\widehat{\mathcal{D}_0(p)}))$ by $pr_{D_{i-1}} \circ \partial_i$ can be computed from the local intersection numbers of $B_t^j(p)$ and $\mathcal{A}_0(q) \cap \Sigma_1$ and the elements of $\pi_1 X$ determined by the perturbed flows by $h$ from $p$ to $q$, which are invariant on $t$ for $j < k$.
Hence from the computation of the case where $h = id$, we obtain
\[ \partial_i^f(\widehat{\mathcal{D}_0(p)}) \sim_k \boldsymbol{K}_i(\widehat{\mathcal{D}_0(p)}) \]
for all $p$ with index $i$, and $K_i$ is non-singular.
Again from Lemma \ref{lem_T} we analogously see the desired relation.
\end{proof}

From the proofs of Lemma \ref{lem_p3} and \ref{lem_p5}, if we choose appropriate $T_1$ and $H$, then the conclusions of the lemmas simultaneously hold for any positive integer $k$, and so
\[ \prod_{i=1}^d [\det(I - \phi_{i-1}) \det K_i]^{(-1)^i} \sim_k [\rho_*(\zeta_f)] [\tau_{\rho}^{Nov}(f)]. \]
Now we can establish Theorem \ref{thm_M} at once from Lemma \ref{lem_p1} and \ref{lem_p2}. \\

\noindent
\textbf{Acknowledgement.}
The author wishes to express his gratitude to Toshitake Kohno for his encouragement and helpful suggestions.
He is greatly indebted to Andrei Pajitnov for a thorough explanation of the deduction of the theorem from his results and for many stimulating conversations.
He would also like to thank Hiroshi Goda, Takayuki Morifuji, Takuya Sakasai and Yoshikazu Yamaguchi for fruitful discussions and advice.
This research was supported by JSPS Research Fellowships for Young Scientists.


\end{document}